\documentclass[11pt]{article}
\usepackage[margin=35mm]{geometry}

\usepackage{amsmath,amsfonts,amsthm,amstext,amssymb,mathtools}
\usepackage{bbm}
\usepackage{comment}
\usepackage{hyperref}
\usepackage{authblk}
\usepackage{thm-restate}

\usepackage[dvipsnames]{xcolor}
\hypersetup{
  linkcolor  = Bittersweet,
  citecolor  = MidnightBlue,
  urlcolor  = MidnightBlue,
  colorlinks = true
}

\let\emptyset\varnothing

\newcommand{\cF}{\mathcal{F} }
\newcommand{\nc}{\overline{N}}
\newcommand{\pr}{\mathbb{P} }
\newcommand{\E}{\mathbb{E} }
\newcommand{\N}{\mathbb{N} }
\newcommand{\R}{\mathbb{R} }
\newcommand{\mis}{\operatorname{MIS} }

\newtheorem{theorem}{Theorem}[section]
\newtheorem{conjecture}{Conjecture}[section]
\newtheorem{lemma}[theorem]{Lemma}

\newtheorem{question}[theorem]{Question}
\newtheorem{proposition}[theorem]{Proposition}

\newtheorem{claim}{Claim}
\newtheorem*{claim*}{Claim}

\theoremstyle{definition}
\newtheorem{definition}[theorem]{Definition}
\newtheorem{remark}[theorem]{Remark}

\newtheorem{observation}{Observation}
\newtheorem{fact}{Fact}

\author[1]{Joshua Cooper\thanks{cooper@math.sc.edu}}
\author[1]{Isaiah Hollars\thanks{isaiah.hollars@sc.edu}}

\affil[1]{University of South Carolina}

\title{Hitting all maximal independent sets in $c$-hollow graphs 
}
\date{\today}

\begin{document}
\maketitle

\begin{abstract}
   Fix a constant $c$ with $0<c<1$. We say a graph $G$ on $n$ vertices is $c$-hollow if every maximal independent set of $G$ has size at least $cn$. Denote by $\tau(G)$ the size of a smallest set of vertices $T\subseteq V(G)$ such that every maximal independent set in $G$ intersects $T$, i.e., $T$ is a transversal for the family of maximal independent sets. In 1991, Bollob\'{a}s, Erd\H{o}s, and Tuza conjectured that if $G$ is $c$-hollow, then $\tau(G)=o(n)$. Using a random construction, we show there exist $c$-hollow graphs with $\tau(G)=\Omega\left(\frac{n^{1/3}}{\log n }\right)$, establishing the first nontrivial lower bound constraining the conjecture and complementing a closely related lower bound due to Alon for {\em maximum} independent sets. We also show the conjecture holds in a strong form for the class of cographs and split graphs. 

\end{abstract}

\section{Introduction}
\subsection{Background}

Throughout, all graphs will be finite, simple, and undirected.  
\begin{definition}
A collection of vertices $I\subseteq V(G)$ is an \textit{independent set} if no edge of $G$ has both its ends in $I$. An independent set $I$ is \textit{maximal} if $I$ is not properly contained in a larger independent set. We use MIS as shorthand for maximal independent set and define \[
\mis(G)=\{I\subseteq V(G): \text{ $I$ is a maximal independent set of $G$}\}.
\] 
\end{definition}

\begin{definition}
    Let $0<c<1$. We say a graph $G$ on $n$ vertices is \textit{$c$-hollow} if $|I|\geq cn$ for all $I\in \mis(G)$. 
\end{definition}

\begin{definition}
Let $\mathcal{F}\subseteq 2^X$ be a collection of subsets of a set $X$. A set $T\subseteq X$ is a \textit{transversal}  (also called a \textit{hitting set} or a \textit{piercing set}) for $\mathcal{F}$ if $F\cap T \neq \emptyset$ for all $F\in \mathcal{F}$. The \textit{transversal number} $\tau(\mathcal{F})$ denotes the size of a smallest transversal for $\mathcal{F}$. Note that $\mis(G)$ is a set system with ground set $V(G)$. By a slight abuse of notation, we define\[
		\tau(G)\coloneqq \tau(\mis(G))
		\]
		to be the size of a smallest transversal for $\mis(G)$. 
\end{definition}

The following conjecture is the focus of this paper. 
\begin{conjecture}[Bollob\'{a}s, Erd\H{o}s, Tuza 1991]\label{conj-maximal}
Let $0<c<1$ be a constant. If $G$ is a $c$-hollow graph on $n$ vertices, then $\tau(G)=o(n)$. 
\end{conjecture}

This question appears to have been first raised in 1991 by Bollob\'{a}s, Erd\H{o}s, and Tuza in \cite{erdos3problems} and was later discussed in \cite[p.52]{chung1998erdos}, although the terminology ``$c$-hollow'' is not used there. Similar problems were also proposed in the 1992 paper \cite{erdos1992covering} of Erd\H{o}s, Tuza, and Gallai. Unpacking the $o(n)$ notation, the conjecture is equivalent to the statement: for every $c\in (0,1)$ and every $\epsilon>0$, there exists some $N=N(c,\epsilon)$ such that if $G$ is a $c$-hollow graph on $n\geq N$ vertices, then $\tau(G)\leq \epsilon n$. So, it is conjectured that large $c$-hollow graphs admit sublinear-sized (in $n$) transversals for $\mis(G)$.

We now give a few remarks illustrating why one might expect small MIS transversals in general for $c$-hollow graphs.

\begin{remark}
Take $n\equiv 0 \pmod 3$ and consider $G=\frac{n}{3}K_3$, the disjoint union of $\frac{n}{3}$ triangles. An independent set in $G$ contains at most one vertex per triangle. Each MIS of $G$ contains exactly one vertex per triangle. Moreover, each choice of one vertex per triangle yields a distinct MIS, so $|\mis(G)|=3^{n/3}$. It was shown independently by Moon and Moser in 1965 (\cite{moon1965}) and Miller and  Muller in 1960 (\cite{miller1960problem}) that the graph $G$ is the unique maximizer of $|\mis(G)|$ among all $n$-vertex graphs (with a slight alteration to the construction when $n\not\equiv 0\pmod 3$). The graph $G$ is $c$-hollow with $c=\frac{1}{3}$ since $|I| = \frac{n}{3}$ for all $I\in \mis(G)$. If $T$ consists of the 3 vertices of a triangle, then $I\cap T \neq \emptyset$ for all $I\in \mis(G)$. Hence, $T$ is a hitting set for $\mis(G)$ of size 3, which is certainly $o(n)$. Note that, more generally, if we consider $G = \frac{n}{t} K_t$ for $t > 3$ (and $t | n$), then $G$ is $1/t$-hollow and $\tau(G) = t$.  %More generally, for any graph $G$ and any $v\in V(G)$, the set $T=\{v\}\cup N(v)$ forms a transversal for $\mis(G)$. 
\end{remark}

\begin{remark}
Write $P_n$ for a path on $n$ vertices, and $G := \overline{P_n^t}$ for the complement of its $t$-th power, i.e., $V(G) = [n]$ and $ij \in E(G)$ whenever $|i-j| > t$.  Then, if $t = cn \in \mathbb{Z}$, it is straightforward to see that $G$ is $c$-hollow, but $\tau(G) \leq n/t  = 1/c$, because we may take $T = t \mathbb{Z} \cap [n]$ as a transversal.  As in the previous example, $\tau(G)$ is not only $o(n)$, but a constant.
\end{remark}

\begin{remark}
    Note that the fractional transversal number {\em does} satisfy $\tau^*(G) \leq 1/c$ for any $c$-hollow graph $G$.  The fractional transversal number is defined to be the minimum value of $\sum_{v \in V(G)} x_v$ subject to the constraints that $x_v \geq 0$ for every $v$ and $\sum_{v \in I} x_v \geq 1$ for each maximal independent set $I \subseteq V(G)$.  In a $c$-hollow graph, we may take $x_v = 1/(cn)$ for every $v$.
\end{remark}

We know of no published results directly concerning Conjecture~\ref{conj-maximal}. However, the following closely related conjecture by Bollob\'{a}s, Erd\H{o}s, and Tuza in the early 1990s has received recent attention and is discussed in \cite{chung1998erdos, erdos3problems}. The independence number of $G$, denoted $\alpha(G)$, is the size of a largest independent set of $G$. We say that $I\subseteq V(G)$ is a \textit{maximum independent set} if $|I|=\alpha(G)$.

\begin{conjecture}[Bollob\'{a}s, Erd\H{o}s, and Tuza]\label{conj-maximum} 
   Let $0<c<1$ be a constant. Let $h(G)$ denote the size of a smallest hitting set for the collection of \textit{maximum} independent sets of $G$. If $G$ is a graph on $n$ vertices with $\alpha(G)\geq cn$, then $h(G)=o(n)$.
\end{conjecture}

We remark that this ``maximum conjecture'' is neither a direct strengthening nor weakening of Conjecture~\ref{conj-maximal}. On the one hand, the maximum conjecture is only concerned with hitting the maximum independent sets, which is a strict subset of $\mis(G)$ in general. On the other hand, the hypothesis in Conjecture~\ref{conj-maximal} that $G$ is $c$-hollow is much stronger than only assuming $\alpha(G)\geq cn$. In 2021, Alon (\cite{alon2021hitting}) observed that an old result of Hajnal (\cite{hajnal1965theorem}) implies $h(G)= 1$ when $\alpha(G)>\frac{n}{2}$. Combining Hajnal's result with the container method, Alon showed that if $G$ is regular and $\alpha(G)\geq (\frac{1}{4}+\epsilon)n$, then $h(G)=O(\sqrt{n\log n})$. Alon also gave a construction with $h(G)=\Theta(\sqrt n)$ and $\alpha(G)\geq n/4$; this is the largest known value for $h(G)$ when $\alpha(G) = \Omega(n)$. Significant progress on this question occurred in the next few years. First, Hajebi, Li, and Spirkl (\cite{hajebi2024hitting}) showed that if $G$ contains no induced $P_5$, then $h(G)$ is bounded above by a function of the clique number $\omega(G)$. Next, Ai, Liu, Xu, and Zhou (\cite{ai2024piercing}) showed that if $G$ contains no induced matching of size $t$, then $h(G)\leq \omega(G)^{3t-3+o(1)}$. Cheng, Huang, Rong, and Xu (\cite{cheng2024sublinear}) showed Conjecture~\ref{conj-maximum} holds in several geometric graph families satisfying a certain ``locally sparse'' condition. In particular, they showed $h(G)=O(\frac{n}{\log n})$ for even-hole-free graphs with $\alpha(G)=\Omega(n)$. As a follow-up, Cheng and Xu showed in \cite{cheng2025bollobas} that $h(G)=o(n)$ for $K_{s,t}$-free graphs with $\alpha(G)=\Omega(n)$ using an elegant probabilistic argument. It is unclear if the methods of these publications can illuminate  Conjecture~\ref{conj-maximal}, since the arguments are tailored to analyze \textit{maximum} independent sets.

\subsection{Results and organization}

Our main result is a lower bound construction constraining Conjecture~\ref{conj-maximal}. 

\begin{theorem}\label{thm-janson-construction}
    Let $0<c<1$ be a constant. There exists a $c$-hollow graph $G$ on $n$ vertices with $\tau(G)=\Omega\left(\frac{n^{1/3}}{\log n }\right)$. 
\end{theorem}
Therefore, the $o(n)$ in Conjecture~\ref{conj-maximal} cannot be made asymptotically lower than $\frac{n^{1/3}}{\log n }$. The idea behind the construction is to build a graph $G$ where $|\mis(G)|$ is large and $\mis(G)$ resembles a uniformly distributed random set system. Alon showed in \cite{alon1990transversal} that $k$-uniform hypergraphs with uniformly random edges have large transversal numbers. Unlike Alon's hypergraphs, the random MIS's in our construction will only be approximately independent, enabling the application of Janson's inequalities (\cite{janson1998new}). 

%\begin{remark}
%It seems to us that Conjecture~\ref{conj-maximal} is not significantly easier to prove when $c$ is large (close to 1). The fact that our construction works for any $c\in (0,1)$ is evidence of this. On the other hand, Conjecture~\ref{conj-maximum} seems easier to prove for larger $c$; e.g., when $c>1/2$ Conjecture~\ref{conj-maximum} trivially holds (as previously discussed). 
%\end{remark}

Next, we prove Conjecture~\ref{conj-maximal} holds in a strong form for cographs and split graphs. A graph is a \textit{cograph} if it contains no induced path on 4 vertices. Disjoint unions of cliques and complete multipartite graphs are examples of cographs. We denote by $i(G)$ the size of a smallest MIS of a graph $G$, i.e., $c = i(G)/n$ is the largest value so that $G$ is $c$-hollow. Equivalently, $i(G)$ is the size of a smallest independent dominating set of $G$. In 1990, Tuza proved a strong upper bound on $\tau(G)$ for some graphs arising as complements of chordal graphs. We have restated Tuza's theorem, as it was originally formulated in terms of cliques.
\begin{theorem}[Tuza \cite{tuza1990covering}]
    If $\overline{G}$ is strongly chordal, then $\tau(G)\leq \frac{n}{i(G)}$. Moreover, $\tau(G)\leq \frac{n}{i(G)}$ if $\overline{G}$ is chordal and $i(G)=3$. 
\end{theorem}

Our result for cographs is analogous to Tuza's theorem. 
    \begin{restatable*}{corollary}{restatecograph}
		\label{corollary-cograph}
	 If $G$ is a cograph, then $\tau(G)\leq \frac{n}{i(G)}$. In particular, if $G$ is a $c$-hollow cograph, then $\tau(G)\leq \frac{1}{c}$. 
	\end{restatable*}

Lastly, we consider split graphs. A graph $G$ is a \textit{split graph} if $V(G)$ can be partitioned into a clique and an independent set.

\begin{restatable*}{theorem}{restatesplit}
		\label{thm-splitgraph}
Let $0<c<1$ and let $n$ be sufficiently large. If $G$ is a $c$-hollow split graph on $n$ vertices, then $\tau(G)\leq \frac{-1}{\log(1-c)}\log n+1$. Moreover, there exist $c$-hollow split graphs with $\tau(G) \geq (1-o(1))\frac{-1}{\log(1-c)}\log n$.
\end{restatable*}

The upper bound follows by constructing a transversal using a greedy heuristic. The lower bound construction is a random split graph analogous to the random uniform hypergraphs in \cite{alon1990transversal}.

The rest of the paper is organized as follows. Section~\ref{sec-preliminaries} contains preliminary definitions. In Section~\ref{sec-hollowify}, we establish a key lemma for Theorem~\ref{thm-janson-construction}, which we then prove in Section~\ref{sec-mainproof}. Our results on cographs and split graphs are in Sections~\ref{sec-cographs} and \ref{sec-splitgraphs} respectively. We conclude in Section~\ref{sec-discussion} with open questions and future directions.

\section{Preliminaries}\label{sec-preliminaries}

Let $G$ be a graph and let $A, B\subseteq V(G)$. It will be helpful to refer to the set of common non-neighbors of $A$ in $B$. 
\begin{definition}
We denote by $\nc_B(A)$ the set of common non-neighbors of $A$ in $B$. Explicitly, \[
    \nc_B(A)\coloneqq \{v\in B: \forall a\in A, \,va\not\in E(G)\}.
    \]
In particular, if $A$ and $B$ are independent sets, then $A\cup \nc_B(A)$ is also independent. If $A=\{v\}$ is a single vertex, we omit the set brackets and write $\nc_B(v)$. 
\end{definition}

\begin{definition}
    Let $G_1$ and $G_2$ be graphs. The \textit{disjoint union} of $G_1$ and $G_2$, denoted $G_1 \cup G_2$, is the graph with vertex set $V(G_1)\sqcup V(G_2)$ and edge set $E(G_1)\sqcup E(G_2)$. The \textit{join} of $G_1$ and $G_2$ is the graph obtained by adding all edges between $V(G_1)$ and $V(G_2)$ to the graph $G_1\cup G_2$. We use $G_1+G_2$ to denote the join of $G_1$ and $G_2$. 
\end{definition}

\begin{definition}
    Given a graph $G$, the parameter $i(G)= \min\{|I|: I\in \mis(G)\}$ is defined as the size of a smallest MIS of $G$. Equivalently, $i(G)$ is the size of a smallest independent dominating set of $G$, which is where the notation originates. See \cite{goddard2013independent} for a survey on this parameter. 
\end{definition}
\begin{definition}
    Given a graph $G$ on $n$ vertices, we define the parameter $\beta(G)=\frac{\tau(G)i(G)}{n}$. The parameter $\beta(G)$ behaves nicely under disjoint unions and joins (see Section~\ref{sec-cographs}) and it allows for a reformulation of Conjecture~\ref{conj-maximal} (see Section~\ref{sec-discussion}).  
\end{definition}

In our constructions for Theorem~\ref{thm-janson-construction} and \ref{thm-splitgraph}, we employ standard probabilistic arguments. We refer the reader to \cite{alon2016probabilistic} for more explanations and examples of probabilistic techniques in combinatorics. Given a sequence of events $\{A_n\}_{n\in \N}$ in associated probability spaces $(\Omega_n,\mathcal{A}_n,\pr_n)_{n\in \N}$, we say that $A_n$ holds with high probability (whp) if $\pr_n(A_n)=1-o_{n\to\infty}(1)$. When referring to random variables $X_1, X_2, \dots, X_n$, we use \textit{iid} as shorthand for independent and identically distributed. We will need the following special case of Hoeffding's inequality.
\begin{theorem}[Hoeffding 1963 \cite{hoeffding1963probability}]
    Let $X=X_1+X_2+\dots+X_n$ be the sum of $n$ iid Bernoulli random variables. Then for any positive real $\lambda>0$, 
    \[
    \pr\bigg[\pm (X-\E[X]) \geq \lambda \bigg]\leq \exp\left(\frac{-2\lambda^2}{n}\right).
    \]
\end{theorem}
Theorem~\ref{thm-janson-construction} and Theorem~\ref{thm-splitgraph} both use the following random graph construction.

\begin{definition}
    Let $G_1$ and $G_2$ be graphs and let $0<p<1$. The random graph $G_1 +_p G_2$ is obtained by adding the edges $\{vw : v \in V(G_1), w \in V(G_2)\}$ iid with probability $p$ to the graph $G_1\cup G_2$.
\end{definition}

Given a nondecreasing sequence $(a_n)_{n\in \N}$ converging to $a$, we write $a_n\nearrow a$. For a positive integer $n\in \N_+$, we use $[n]$ to denote the set $\{1,2,\dots, n\}$. To simplify the presentation, we omit floors and ceilings when they are not crucial to the argument. Throughout the paper, $\log(\cdot)$ always refers to the natural logarithm.

\section{A random $c$-hollow graph construction}\label{sec-hollowify}

The goal of this section is to show the construction for Theorem~\ref{thm-janson-construction} is indeed $c$-hollow. Fix a graph $H$. The following lemma shows that, for an appropriately sized independent set $I$ and probability $p$, the random graph $G\coloneqq H+_p I$ is $c$-hollow whp. 

\begin{lemma}\label{lemma-rand-hollowifier}
    Let $H$ be a graph on $n$ vertices with $|\mis(H)|=\exp\left(o(n) \right)$ and $\alpha(H)\to \infty$ as $n\to \infty$. Let $0<c<1$ be a constant. Choose $k\in \N_+$ and $0<\epsilon<1$ so that $\gamma \coloneqq \frac{k+1}{k}\frac{1}{1-\epsilon}c <1$ (e.g., take $k=\lceil\frac{3c}{1-c}\rceil$ and $\epsilon= \frac{1-c}{3}$). Let $I$ be an independent set of size $kn$. Put $p=\frac{\log (1/\gamma)}{\alpha(H)}$. Then $G \coloneqq H+_p I$ is a $c$-hollow graph whp as $n\to\infty$. 
\end{lemma}

\begin{proof}
Notice that each MIS in $G$ is of the form $J\cup \nc_I(J)$ for some independent set $J$ of $H$ (we adopt the convention that the empty set is independent and $\nc_I(\emptyset)=I$). Since $J_1\subseteq J_2\implies |\nc_I(J_2)|\leq |\nc_I(J_1)|$, it suffices to show that the event \[
\bigwedge_{J\in \mis(H)}\{|\nc_I(J)|\geq c|V(G)|\}\]
occurs with high probability. Fix some maximal independent set $J\in \mis(H)$. Let $X$ denote the random variable $|\nc_I(J)|$. Then $X$ is sampled from a binomial distribution with $kn$ trials and success probability \begin{align*}
    (1-p)^{|J|}&= \left(1-\frac{\log (1/\gamma)}{\alpha(H)} \right)^{|J|} \geq \left(1-\frac{\log (1/\gamma)}{\alpha(H)} \right)^{\alpha(H)}\\
    &\to \exp\left(-\log(1/\gamma) \right) = \gamma \quad\text{as $n\to\infty$}.
    \end{align*}
\begin{comment}
    since $(1+\delta/s)^s \rightarrow \exp(\delta)$ if $s \rightarrow \infty$ and $\delta^2/s \rightarrow 0$, and indeed in this case, $s = \alpha(G) \rightarrow \infty$ and
$$
\delta = \log(\gamma) \leq \log c + 1/k + 2 \epsilon = O(1).
$$
\end{comment}
Then $\E[X]\geq  (1-o(1)) \gamma kn$. By Hoeffding's inequality, we have \begin{align*}
    \pr\big[X< (1-\epsilon)\gamma kn\big] &\leq 
    \pr\bigg[ \E[X]-X\geq  (\epsilon-o(1)) \gamma kn \bigg] \\
    &\leq \exp\bigg( -2\big((\epsilon-o(1)) \gamma kn\big)^2/kn \bigg)=\exp\left(-\Omega(n) \right).
\end{align*}
Using that $\gamma = \frac{k+1}{k}\frac{1}{1-\epsilon}c$, we have \[
\frac{(1-\epsilon)\gamma kn}{|V(G)|}=\frac{(1-\epsilon)\gamma kn}{(k+1)n}=\frac{k}{k+1}(1-\epsilon)\cdot \frac{k+1}{k}\frac{1}{1-\epsilon}c=c,
\]
so the bad event $\{|\nc_I(J)|<c|V(G)|\}$ happens with probability at most $\exp\left(-\Omega(n) \right)$. By assumption, $|\mis(H)|=\exp\left(o(n)\right)$. Therefore, applying the union  bound over each $J\in \mis(H)$ gives the result. 
\end{proof}

\section{Proof of Theorem~\ref{thm-janson-construction}}\label{sec-mainproof}

We begin with high-level overview of the proof. Put $m=n^{1/3}$. Let $H$ be the join of $m$ many disjoint unions of $m$ many cliques of size $m$, a graph with $n$ vertices. Consider the random graph $G=H+_p I$, where $p$ is chosen as in Lemma~\ref{lemma-rand-hollowifier} to guarantee that $G$ is $c$-hollow whp.

We prove that the random graph $G$ satisfies $\tau(G)>t\coloneqq \frac{n^{1/3}}{10\log n}$ with high probability. The proof strategy is similar to \cite{alon1990transversal}. We start by fixing an arbitrary set $T\subseteq V(G)$ with $|T|=t$. The goal is to show that, for a random $G$,  the probability that $T$ is a transversal for $\mis(G)$ is $o\left(1/\binom{|V(G)|}{t}\right)$. Therefore, the probability that $\tau(G)\leq t$ tends to 0 by applying the union bound over all possible transversal sets $T\subseteq V(G)$ with $|T|=t$. 

Recall that each MIS in $G$ is of the form $J\cup \nc_I(J)$ for some independent set $J$ of $H$. If $T$ is a transversal for $\mis(G)$, then $T$ is also a transversal for \[
\cF=\cF(T)\coloneqq \{\nc_I(J): J\in \mis(H)\text{ and } T\cap J=\emptyset\}.
\]
(The collection $\cF$ used in the proof is actually a subset of the one above.) Two facts aid us in showing $\pr\left[T\text{ is a transversal for } \cF\right]$ is small. First, by our choice of $H$, we have that $|\cF|$ is relatively large. Second, the collection of events \[\big\{\{F\cap T=\emptyset\}: F\in \cF \big\}\] is \textit{approximately} pairwise independent, since the typical intersection size of independent sets in $\mis(H)$ is small. The combination of these facts allows us to deduce $\pr\left[T\text{ is a transversal for } \cF\right] = o\left(1/\binom{|V(G)|}{t}\right)$ by applying one of Janson's inequalities.  For the sake of readability, some of the more tedious calculations are relegated to Appendix~\ref{sec-appendix}.

\begin{proof}[Proof of Theorem \ref{thm-janson-construction}]
Put $m=n^{1/3}$, and let $H=H_1+H_2+\dots+H_m$ where $H_i\cong mK_m$ for $i\in [m]$. Let $c$ be an arbitrary constant with $0<c<1$. Put $k=\lceil\frac{3c}{1-c}\rceil$ and $\epsilon=\frac{1-c}{3}$ so the hypothesis of Lemma~\ref{lemma-rand-hollowifier} holds. Let $I$ be an independent set on $kn$ vertices. Let $G= H+_p I$, where $p\coloneqq \frac{\log(1/\gamma)}{m}$ and $\gamma \coloneqq \frac{k+1}{k}\frac{1}{1-\epsilon}c<1$. Note that $|\mis(H)| = m^{m+1} = \exp\left(o(n) \right)$, so Lemma~\ref{lemma-rand-hollowifier} implies that $G$ is $c$-hollow whp as $n\to\infty$. For any $J\in \mis(H)$, we have 
\[
\E\left [ |\nc_I(J)| \right ] = kn(1-p)^m = kn\left(1-\frac{\log(1/\gamma)}{m}\right)^m\nearrow \gamma kn.
\]
Fix a set $T\subseteq V(G)$ of size $t= \frac{m}{10\log n}$. Define\[
\cF_T= \{\nc_I(J): \exists i\in [m],\, J\in \mis(H_i)\text{ and } T\cap H_i=\emptyset\}.
\]
As $T$ can intersect at most a $\frac{1}{10\log n}$ fraction of the $H_i$ subgraphs, we have 
\[
|\cF_T|\geq \left(1-\frac{1}{10\log n}\right) |\mis(H)|.
\]
We allow for the collection $\cF_T$ to contain duplicate elements (formally, $\cF_T$ is a multiset), so the quantity $|\cF_T|$ above is counting possible duplicates. If \(T\) is a transversal for $\mis(G)$, then $T$ must be a transversal for $\cF_T$.

Our goal is to provide an upper bound on the probability that $T$ is a transversal for $\cF_T$ using Janson's inequality. Write $\{J_i\}_{i \in \mathcal{I}}$ for an arbitrary indexing of the collection $\{J: \nc_I(J)\in \cF_T\}$ and let $I_i= \nc_I(J_i)$, so that, in this notation, $\mathcal{F}_T= \{I_i\}_{i \in \mathcal{I}}$. Define the set $T^{+}= (T\cap I)\cup A$, where $A\subseteq I$ is arbitrarily (but deterministically) chosen so that $|T^{+}|=t$. Define the indicator random variable $X_i= \mathbbm{1}\{T^{+}\cap I_i=\emptyset\}$. Let $p_i\coloneqq \E[X_i]= \pr\{T^{+}\cap I_i=\emptyset\}$. We desire a strong upper bound for 
\[
\pr \left[ T\text{ is a transversal for }\cF_T \right] \leq \pr\left[ T^{+}\text{ is a transversal for }\cF_T \right] =\pr\left[\sum_i X_i = 0 \right].
\] 
If the $X_i$'s were mutually independent, we would obtain the exponentially small upper bound $\prod_{i}(1-p_i)$. Since we \textit{almost} have independence, we will apply the following (special case of a) theorem of Janson.
\begin{theorem}[special case of Theorem 3 in \cite{janson1998new}]\label{thm-janson} Let $\Gamma$ be a graph with vertex set $\mathcal{I}$ with adjacencies defined by $i\sim j$ whenever $J_i$ and $J_j$ are both contained in a common $H_\ell$ subgraph.  Then
    \[
    \pr\left[\sum_iX_i=0 \right]\leq \exp\left(-\min \left (\frac{\mu^2}{8\Delta}, \frac{\mu}{6\delta}, \frac{\mu}{2} \right )\right),
    \]
\end{theorem}
where 
\begin{itemize}
    \item $\mu\coloneqq \sum_i p_i = \E\left[\sum_i X_i\right]$,
    %\item $i\sim j$ if $ij\in E(\Gamma)$,
    \item $\delta \coloneqq \max_i\left(\sum_{j\sim i}p_j\right)$,
    \item $\Delta\coloneqq \sum_{\{i,j\}: i\sim j}\E[X_iX_j] = \frac{1}{2}\sum_i \sum_{j\sim i}\E[X_iX_j]$.
\end{itemize}
One can easily check that $X_i$ is mutually independent of the collection $\{X_j: j\not\sim i\}$, a requirement of the choice of graph in Janson's inequality. Our notation follows \cite{janson1998new} with the exception that we use $X_i$ (instead of $I_i$) for our indicator random variables. 
\begin{claim*}
    $\frac{\mu^2}{\Delta}$, $\frac{\mu}{\delta}$, and $\mu$ are all at least $(1-o(1))m$.
\end{claim*}
Assuming the claim holds, Theorem~\ref{thm-janson} implies 
$$
\pr\left[T^{+}\text{ is a transversal for }\cF_T \right] \leq \exp\left(-(1-o(1))m/8\right).
$$
Note that 
\begin{align*}
    \binom{(k+1)n}{t}\leq \left((k+1)n\right)^t= \exp\left(t\log((k+1)n)\right)
&= \exp\left(\frac{m}{10 \log n} (\log (k+1) + \log n) \right) \\
&\leq \exp\left(m/9 \right).
\end{align*}

Therefore, the probability that a randomly sampled graph $G$ admits a transversal of size $t$ is at most $\exp\left((\frac{1}{9} -\frac{1}{8} +o(1))m\right)\to 0$ by applying the union bound over all choices of $T\subseteq V(G)$ with $|T|=t$. Since $|V(G)|=(k+1)n$ and $k$ is a fixed constant,  $t=\frac{n^{1/3}}{10\log n}=\Omega\left(\frac{|V(G)|^{1/3}}{\log |V(G)|}\right)$. Hence, the graph $G$ satisfies the requirement of Theorem~\ref{thm-janson-construction} whp. It now suffices to verify the claim, which we break into 3 parts. 

Let $v\in T^{+}$ and $i \in \mathcal{I}$. Then 
\[
\pr \{v\in I_i\} = \pr \{v\in \nc_I(J_i)\} = (1-p)^m\nearrow  \gamma.
\]
The collection of events $\big\{ \{v\in I_i\}: v\in T^{+}\big\}$ are mutually independent. Hence, for all $i\in \mathcal{I}$, 
\begin{align*}
p_i=\pr\left[\bigwedge_{v\in T^{+}} \{v\not\in I_i\}\right]=\prod_{v\in T^{+}} (1-(1-p)^m)
= \left(1-(1-p)^m\right)^t.
\end{align*}

\begin{claim}
    $\mu$ is at least $(1-o(1))m$. 
\end{claim}

Since $(1-p)^m\leq \gamma$, we have  $p_i\geq 
(1-\gamma)^t$. 
We then deduce \[\mu= \sum_ip_i \geq \left(1-\frac{1}{10\log n}\right) |\mis(H)|(1-\gamma)^{t} = \left(1-o(1)\right)m^{m+1}(1-\gamma)^{t}.\] 
As $t=o(m)$, this is more than enough to prove the claim. \hfill $\blacksquare$

\begin{claim}
    $\frac{\mu}{\delta}$ is at least $(1-o(1))m$.
\end{claim} 

Given $J_i\in \mis(H)$, there are at most $m^m$ other $J_j$ with $j\sim i$. Hence $\delta= \max_i\left(\sum_{j\sim i}p_j\right)\leq  m^m p_j$. We have \begin{align*}
    \frac{\mu}{\delta}&\geq \frac{\left(1-\frac{1}{10\log n}\right)m^{m+1} p_j}
    {m^m p_j}
    = \left(1-o(1)\right)m.
\end{align*}

\hfill $\blacksquare$

\begin{claim}
    $\frac{\mu^2}{\Delta}$ is at least $(1-o(1))m$.
\end{claim}
For $i,j \in \mathcal{I}$, define $p_{ij}\coloneqq \pr\{v\in I_i\cup I_j\}$ (the choice of $v\in T^+$ does not affect this quantity). Since, for fixed $i,j \in \mathcal{I}$, the collection of events $\big\{\{v\in I_i\cup I_j\}: v\in T^{+} \big\}$ are mutually independent, we have \begin{align*}
    \E[X_iX_j] &= \pr\left \{ T^{+}\cap (I_i \cup I_j) = \emptyset \right \}
    = \pr \left[ \bigwedge_{v\in T^{+}}\{v\notin I_i\cup I_j\}\right]
    = \prod_{v\in T^{+}}(1-p_{ij})
    = (1-p_{ij})^t.
\end{align*}
To estimate $p_{ij}$, put $r=|J_i\cap J_j|$ and write 
\begin{align*}
    p_{ij} &= \pr \left\{ v\in I_i\cup I_j\right\}
    = \pr \left[ \{v\in I_i\}\vee \{v\in I_j\} \right]\\
    &= \pr \{v\in I_i\} + \pr  \{v\in I_j\}-\pr\left[\{v\in I_i\}\wedge \{v\in I_j\}\right]\\
    &=2(1-p)^m - (1-p)^{|J_i\cup J_j|} \\
    &= 2(1-p)^m - (1-p)^{2m-r}.
\end{align*}
Then $\E[X_iX_j]=\left(1-2(1-p)^m+(1-p)^{2m-r}\right)^t$. 
Fix an $\ell \in [m]$ and $J'\in \mis(H_\ell)$. For $r\in \{0,1,\dots,m\}$, define the constants \[
N_r=|\{J\in \mis(H_\ell): |J\cap J'|=r\}| = \binom{m}{r}(m-1)^{m-r}.
\]
Putting everything together, we have \begin{align*}
\Delta= \frac{1}{2}\sum_i\sum_{j\sim i}\E[X_iX_j] &\leq \frac{1}{2} |\mis(H)| \max_i \sum_{j\sim i}\E[X_iX_j]\\
&\leq  |\mis(H)| \sum_{r=0}^m N_r\left(1- 2(1-p)^m + (1-p)^{2m}(1-p)^{-r}\right)^t\\
&\leq (1+o(1)) |\mis(H)| \sum_{r=0}^m N_r\left(1- 2\gamma + \gamma^2(1-p)^{-r}\right)^t\quad\text{by Appendix Fact~\ref{fact2}}.
\end{align*}

Estimating $\frac{\mu^2}{\Delta}$ gives \begin{align*}
    \frac{\mu^2}{\Delta} &\geq \frac{\left(1-\frac{1}{10\log n}\right)^2|\mis(H)|^2(1-\gamma)^{2t}}{(1+o(1))|\mis(H)|  \sum_{r=0}^m N_r\left(1- 2\gamma + \gamma^2(1-p)^{-r}\right)^t}\\
    &=(1-o(1))\frac{ m^{-m} |\mis(H)|(1-2\gamma+\gamma^2)^t}{\sum_{r=0}^{m} m^{-m} N_r (1-2\gamma+\gamma^2(1-p)^{-r})^t}\\
    &\geq (1-o(1))\frac{ m(1-2\gamma+\gamma^2)^t}{\sum_{r=0}^{m}\frac{1}{e}\binom{m}{r}(m-1)^{-r} (1-2\gamma+\gamma^2(1-p)^{-r})^t},
\end{align*}
since $m^{-m}|\mis(H)|=m$ and $m^{-m}N_r=\frac{\binom{m}{r} (m-1)^{m-r}}{m^m} \leq \frac{1}{e}\binom{m}{r}(m-1)^{-r}$. Put $A=(1-2\gamma+\gamma^2)$. Since $\frac{1-2\gamma+\gamma^2(1-p)^{-r}}{A}= 1+ \frac{\gamma^2}{A}((1-p)^{-r}-1)$, after dividing the numerator and denominator by $A^t$, we obtain \[
\frac{\mu^2}{\Delta}\geq \frac{(1-o(1)) m}{\sum_{r=0}^{m}\frac{1}{e}\binom{m}{r}(m-1)^{-r} \left(1+ \frac{\gamma^2}{A}((1-p)^{-r}-1)\right)^t}.
\]
For $m$ sufficiently large, $(1-p)^{-r}\leq C\frac{r}{m}+1$ for some fixed constant $C=C(\gamma)$ (see Appendix Fact~\ref{fact3}). Hence, the denominator is at most \begin{align*}
\frac{1}{e}\sum_{r=0}^{m} \binom{m}{r}(m-1)^{-r} \left( 1+\frac{\gamma^2}{A}\frac{Cr}{m}\right)^{m/10\log n} 
&\leq \frac{1}{e}\sum_{r=0}^{m} \binom{m}{r}(m-1)^{-r} \exp\left( \frac{\gamma^2}{A}\frac{Cr}{m}\frac{m}{10\log n}\right).
\end{align*}
Applying the binomial theorem, this quantity is at most 
\[\frac{1}{e}\left(\frac{1}{m-1}\exp\left( \frac{\gamma^2C}{10A\log n}\right) + 1\right)^m 
\leq \frac{1}{e}\exp\left( \frac{m}{m-1} \exp\left(\frac{\gamma^2C}{10A\log n}\right)\right),
\]
which tends to $1$ as $n\to \infty$. Hence, $\frac{\mu^2}{\Delta}\geq (1-o(1)) m$. 
\hfill$\blacksquare$

 This completes the proof of the claim that $\min\left( \frac{\mu^2}{\Delta}, \frac{\mu}{\delta}, \mu\right) \geq (1-o(1))m$ and also completes the proof of the theorem.  

\end{proof}

\section{Cographs}\label{sec-cographs}
	There are many equivalent definitions for cographs (e.g., cographs are the class of induced $P_4$-free graphs). For our purposes, the following definition will be easiest to use. 
	\begin{definition}
		A \textit{cograph} is a graph obtained by the following recursive construction.
		\begin{enumerate}
			\item $K_1$ is a cograph.
			\item The disjoint union $G\cup H$ of two cographs is a cograph.
			\item The join $G+H$ of two cographs is a cograph. 
		\end{enumerate}
		In other words, the class of cographs is the smallest graph class containing $K_1$, which is closed under disjoint unions and joins. 
	\end{definition}

Recall that $i(G)$ denotes the size of a smallest maximal independent set of $G$ and $\beta(G)\coloneqq  \frac{\tau(G)i(G)}{n}$. The next proposition describes how these parameters interact under disjoint union and join. We will use the proposition to show that cographs satisfy $\beta(G)\leq 1$, or equivalently, $\tau(G)\leq \frac{n}{i(G)}$. 
	\begin{proposition}\label{prop-params-join-union}
		For $i\in \{1,2\}$, let $G_i$ be a graph on $n_i$ vertices. We shorten $\tau(G_i)$ to $\tau_i$ and use the same shorthand for the other graph parameters. The following table summarizes how $\tau$, $i$, and $n$ behave for $G=G_1\cup G_2$ and $G=G_1+G_2$. 
		\begin{center}
			\begin{tabular}{||c|| c c c||} 
				\hline
				\phantom{x} & $\tau(G)$ & $i(G)$ & $n=|V(G)|$  \\ 
				\hline\hline
				\hline
				$G=G_1\cup G_2$& $\min(\tau_1,\tau_2)$ & $i_1+i_2$ & $n_1+n_2$ \\
				\hline
				$G=G_1+G_2$ & $\tau_1+\tau_2$ & $\min(i_1, i_2)$ & $n_1+n_2$\\
				\hline
			\end{tabular}
		\end{center}
Moreover, for both $G=G_1\cup G_2$ and $G=G_1+G_2$, the following inequality holds: \begin{equation}\label{betainterpolate}
			\beta(G)\coloneqq\frac{\tau(G) i(G)}{n} \leq \frac{n_1}{n_1+n_2}\beta_1 + \frac{n_2}{n_1+n_2}\beta_2.
		\end{equation}
		Equality in \eqref{betainterpolate} holds for $G=G_1\cup G_2$ when $\tau_1=\tau_2$ and holds for $G=G_1+G_2$ when $i_1=i_2$. 
	\end{proposition}
Inequality~\eqref{betainterpolate} can be interpreted as saying that $\beta(G)$ is upper bounded by a strict convex linear combination of $\beta_1$ and $\beta_2$. 
	
\begin{proof}
    Observe that the maximal independent sets of $G_1\cup G_2$ are all formed by taking the union of an MIS from $G_1$ and an MIS from $G_2$. The maximal independent sets of $G_1+G_2$ must lie entirely inside one of the two parts, so $\mis(G_1+G_2)=\mis(G_1)\cup \mis(G_2)$. The table of parameters follows readily from these observations. We now show how to derive \eqref{betainterpolate}. For $G=G_1\cup G_2$, we have \begin{align*}
        \frac{\tau(G) i(G)}{n} &= \frac{\min(\tau_1,\tau_2)(i_1+i_2)}{n_1+n_2}\\
        &= \frac{\min(\tau_1,\tau_2)i_1}{n_1+n_2} + \frac{\min(\tau_1,\tau_2)i_2}{n_1+n_2}\\
        &\leq \frac{\tau_1i_1}{n_1+n_2} + \frac{\tau_2i_2}{n_1+n_2}\quad\text{ with equality when $\tau_1=\tau_2$}\\
        &= \frac{n_1}{n_1+n_2}\beta_1 + \frac{n_2}{n_1+n_2}\beta_2.
    \end{align*}
    The proof is analogous for $G_1+G_2$. 
\end{proof}

\restatecograph

	\begin{proof}
	By inequality~\eqref{betainterpolate} in Proposition~\ref{prop-params-join-union}, $\beta(G_1\cup G_2)\leq \max (\beta_1,\beta_2)$ and $\beta(G_1+G_2)\leq \max (\beta_1,\beta_2)$.  Observe that $\beta(K_1)= \frac{\tau(K_1)i(K_1)}{1}=1$. Since all cographs can be recursively built up from $K_1$ by disjoint unions and joins, we deduce $\beta(G)\leq 1$ holds for any cograph $G$. Rearranging $\beta(G)\leq 1$ gives the result. 
	\end{proof}
	\begin{remark}\label{remark-cographlike}
		We can create slightly more general graph classes by building up from graphs other than $K_1$. Let $\mathcal{G}$ be the smallest class of graphs containing $\{G_1,\dots,G_n\}$ which is closed under disjoint union and join. Then for any $c$-hollow graph $G\in \mathcal{G}$, we have $\tau(G)\leq \frac{\max_i\{\beta_i\}}{c}$.
	\end{remark}

\section{Split graphs}\label{sec-splitgraphs}

A graph $G$ is a \textit{split graph} if there exists a partition $K\cup I$ of $V(G)$ such that $K$ induces a clique and $I$ induces an independent set. We establish a sharp upper bound for $\tau(G)$ in $c$-hollow split graphs. 

\restatesplit

\begin{observation}\label{obs-splitgraphtau}
	Let $G=K\cup I$ be a split graph with $K$ and $I$ nonempty. Then $\mis(G)$ consists of all sets $\{v\}\cup \nc_I(v)$ for $v\in K$, and  additionally, $\mis(G)$ contains $I$ if $\nc_K(I)=\emptyset$. Let $T$ be a transversal for $\mis(G)$ with $|T| = \tau(G)$. Assume $G$ has no dominating vertices (otherwise $G$ is not $c$-hollow when $n>1/c$). Each $v\in K$ only belongs to one MIS, so if $T$ contains any $v\in K$, we could replace $v$ with any $u\in \nc_I(v)$ to obtain a new transversal $(T\setminus\{v\})\cup \{u\}$ for $\mis(G)$. By iterating this procedure, we obtain a transversal for $\mis(G)$ of size $\tau(G)$ consisting only of vertices from $I$. Define the set system $\mathcal{F}= \{\nc_I(v): v\in K\}$. Then $\tau(\mathcal{F}) = \tau(G)$. 
\end{observation}

The following lemma will be used to establish the upper bound in Theorem~\ref{thm-splitgraph}.  Versions of this lemma have appeared in many contexts in the literature, perhaps most notably, \cite{lovasz75}.  We include the proof here for completeness and because we need a precise constant in the statement.

\begin{lemma} \label{lem-transversal-for-thickF}
Fix a constant $0<c<1$ and let $\cF\subseteq2^{[n]}$ be a collection of subsets of $[n]$ where $|F|\geq cn$ for all $F\in \cF$. Then \[
\tau(\mathcal{F}) \leq \frac{-1}{\log (1-c)}\log|\mathcal{F}|+1.
\]
\end{lemma}

\begin{proof}
We build a transversal $T\subseteq [n]$ using the greedy algorithm. Initialize $T=\emptyset$. While $\mathcal{U}= \{F\in \cF: F\cap T=\emptyset\}$ is nonempty, choose $j\in [n]$ to maximize $|\{F\in \mathcal{U}: j\in F\}|$, update $T$ to $T\cup\{j\}$. In other words, we iteratively add to $T$ the vertex $j\in [n]$ which hits the most unhit edges of $\mathcal{F}$. The key observation is that, at each step, we add some $j\in [n]$ which hits at least a $c$-proportion of the unhit edges of $\cF$. This follows from a simple averaging argument on the incidence bipartite graph for $\mathcal{F}$. Therefore, the greedy algorithm builds a transversal of size at most $t+1$, where $t$ satisfies $(1-c)^t|\mathcal{F}|=1$. Solving for $t$ gives $t = -\log |\mathcal{F}|/\log(1-c)$, which establishes the lemma.
\end{proof}

\begin{proposition}
Suppose $0 < c < 1$ and $n$ is sufficiently large.  Let $G$ be a $c$-hollow split graph on $n$ vertices. Then $\tau(G)\leq \frac{-1}{\log (1-c)}\log n+1$.
\end{proposition}
	\begin{proof}
Let $G=K\cup I$ be a $c$-hollow split graph on $n$ vertices. We may assume $|K|\geq \frac{-1}{\log (1-c)}\log n$. Otherwise, we are done since $K\cup\{u\}$ (for $u\in I$ chosen arbitrarily) is a transversal. As in Observation~\ref{obs-splitgraphtau}, define $\cF= \{\nc_I(v): v\in K\}$ and note $\tau(\cF)=\tau(G)$. Since $G$ is $c$-hollow, we have $|F|\geq cn-1$ for all $F\in \cF$. One checks that \[
\frac{|F|}{|I|}\geq \frac{cn-1}{n-\frac{-1}{\log (1-c)}\log n}\geq c
\]
for $n$ sufficiently large. By Lemma~\ref{lem-transversal-for-thickF}, $\tau(\cF)\leq \frac{-1}{\log (1-c)}\log n+1$.
\end{proof}

This establishes the upper bound in Theorem~\ref{thm-splitgraph}. Next, we address the lower bound. By Observation~\ref{obs-splitgraphtau}, the problem reduces to finding a construction where $\tau(\cF)$ is large and $|F|\geq cn$ for $|F|\in \cF$. This allows us to apply an argument similar to one appearing in \cite{alon1990transversal} to prove the following.
\begin{proposition}\label{prop-randomsplit-largetau}
	Let $0<c<1$. There exists a $c$-hollow split graph on $n$ vertices with \[\tau(G)\geq (1-o(1)) \frac{-1}{\log (1-c)} \log n.\] 
\end{proposition}

The proof (see Appendix~\ref{sec-appendix}) closely follows that in \cite{alon1990transversal} and is essentially a special case of the analysis for Theorem~\ref{thm-janson-construction}. The idea is to consider a random split graph $G=K+_p I$. We remark that similar problems are also analyzed in \cite{telelis2005absolute, vercellis1984probabilistic}. Interestingly, it is shown in \cite{telelis2005absolute} that a simple algorithm (with better runtime than greedy) performs within a $(1+o(1))$-factor of the greedy algorithm whp when the set system is appropriately random, as is the case here.

\section{Questions}\label{sec-discussion}

Conjecture~\ref{conj-maximal} remains wide open in general. The construction in Theorem~\ref{thm-janson-construction} shows that one cannot hope to show $\tau(G)=o\left(n^{1/3-\epsilon} \right)$ for general $c$-hollow graphs. In the next two subsections, we propose further directions of study. 

\subsection{Graphs with large and small values of $\beta(G)$ }
Recall that $i(G)$ is the size of a smallest MIS in $G$ and $\beta(G)\coloneqq \frac{\tau(G)i(G)}{n}$. Conjecture~\ref{conj-maximal} can be reformulated as follows. 
\begin{conjecture}[Reformulation of Conjecture~\ref{conj-maximal}]
    For all graphs $G$ on $n$ vertices, $\beta(G)=o(n)$. Note that this trivially holds when $i(G)=o(n)$, i.e., when $G$ is not $c$-hollow. 
\end{conjecture}
	
This leads naturally to the following question. 
	\begin{question}
	What is the maximum value of $\beta(G)$ over all $n$-vertex graphs $G$?  
	\end{question}
The largest we have managed to make $\beta(G)$ is $\Omega\left(\frac{n^{1/3}}{\log n}\right)$, which is achieved by the graph in Theorem~\ref{thm-janson-construction}. Recall that for cographs, we have $\beta(G)\leq 1$. Another interesting direction would be to find other graph classes where $\beta(G)$ is small. 
\begin{question}
What are some natural graph classes $\mathcal{G}$ with the property that for all $G\in \mathcal{G}$, $\beta(G)\leq C$ for some universal constant $C>0$?
\end{question}
Similar questions have been investigated in \cite{erdos1992covering,tuza1990covering}. Tuza showed in \cite{tuza1990covering} that $\beta(G)\leq 1$ holds for complements of strongly chordal graphs. Chordal graphs, on the other hand, do not have this property; the random split graphs (a subclass of chordal graphs) in Theorem~\ref{thm-splitgraph} have $\beta(G)=\Omega(\log n)$. 
	
	\subsection{The strong conjecture}
Conjecture~\ref{conj-maximal} asserts that $c$-hollow graphs admit $o(n)$-sized transversals for $\mis(G)$. Conjecture~\ref{conj-maximum} (the ``maximum conjecture'') asserts that graphs with $\alpha(G)\geq cn$ admit $o(n)$-sized transversals for the collection of maximum independent sets. We tentatively conjecture the following.
	\begin{conjecture}[Strong conjecture]\label{conj-strong}
		Let $0<c<1$ be a constant. All graphs $G$ on $n$ vertices admit $o(n)$-sized transversals for the set of large MIS's, i.e., 
        \[
		\mis_{\text{large}}(G)=\{I\in \mis(G): |I|\geq cn\}.
		\]
	\end{conjecture}
Note that Conjecture~\ref{conj-strong} implies both Conjecture~\ref{conj-maximal} and Conjecture~\ref{conj-maximum}, earning its name. A counterexample to this strong conjecture would be helpful in mapping out the boundaries of what one could hope to prove regarding the other two conjectures.

\appendix
\section{Appendix}\label{sec-appendix}

\begin{fact}\label{fact1}
Let $x\in \R$. Then $\left(1-\frac{x}{m}\right)^m = e^{-x}\left(1-\frac{x^2}{2m}+O\left(\frac{1}{m^2}\right)\right)$ as $m\to\infty$. 
\end{fact}

\begin{proof}
    Put $y=\left(1-\frac{x}{m}\right)^m$. Then \begin{align*}
        \log y&= m \log \left(1-\frac{x}{m}\right)\\
        &= m \left(-\frac{x}{m}-\frac{x^2}{2m^2}-\frac{x^3}{3m^3}-\dots\right)\\
        &=-x -\frac{x^2}{2m}-O\left(\frac{1}{m^2}\right),
    \end{align*}
    so we have \begin{align*}
        y &= \exp\left(-x -\frac{x^2}{2m}-O\left(\frac{1}{m^2}\right) \right)\\
        &= e^{-x}\exp\left(-\frac{x^2}{2m}-O\left(\frac{1}{m^2}\right)\right)\\
        &= e^{-x}\left(1 -\frac{x^2}{2m}-O\left(\frac{1}{m^2}\right) + \frac{1}{2!}\left(-\frac{x^2}{2m}-O\left(\frac{1}{m^2}\right)\right)^2 + \dots \right)\\
        &= e^{-x}\left(1- \frac{x^2}{2m} + O\left(\frac{1}{m^2}\right)\right).
    \end{align*}
\end{proof}

\begin{fact}\label{fact2}
In the proof of Theorem~\ref{thm-janson-construction}, we have
\[
\left(1- 2(1-p)^m + (1-p)^{2m}(1-p)^{-r}\right)^t
\leq (1+o(1)) \left(1- 2\gamma + \gamma^2(1-p)^{-r}\right)^t.
\]
\end{fact}
\begin{proof}
    Using Fact~\ref{fact1} with $x=\log(1/\gamma)$ and cleaning up, we obtain \[
    1-2(1-p)^m+(1-p)^{2m}(1-p)^{-r} = 1-2\gamma+\gamma^2(1-p)^{-r}+O\left(\frac{1}{m}\right).
    \]
    Then the ratio
    \begin{align*}
    R_m&= \frac{\left(1- 2(1-p)^m + (1-p)^{2m}(1-p)^{-r}\right)^t}{\left(1- 2\gamma + \gamma^2(1-p)^{-r}\right)^t} \\
    &= \left(\frac{1- 2\gamma + \gamma^2(1-p)^{-r}+O(1/m)}{1- 2\gamma + \gamma^2(1-p)^{-r}}\right)^t \leq\left(1+ \frac{C}{m}\right)^t 
    \end{align*}
    for some absolute constant $C=C(\gamma)$. As $t=\frac{m}{10\log n}$, $R_m\leq \exp(\frac{C}{m}\frac{m}{10\log n})= (1+o(1))$. 
\end{proof}

\begin{fact}\label{fact3}
    In the proof of Theorem~\ref{thm-janson-construction}, for $m$ sufficiently large, $(1-p)^{-r}\leq C\frac{r}{m}+1$ for some fixed constant $C=C(\gamma)$.
\end{fact}
\begin{proof}
    Write \[
    (1-p)^{-r} = \exp\left(-r \log(1-p)\right) = \exp\left(-m\log\left(1-\frac{\log(1/\gamma)}{m}\right) \frac{r}{m}\right).
    \]
By using the Taylor expansion for $-\log(1-x)$, one can bound $-m\log\left(1-\frac{\log(1/\gamma)}{m}\right)$ from above by some fixed constant $D=D(\gamma)$. So, $(1-p)^{-r}\leq \exp(D\frac{r}{m})$. Note $\frac{r}{m}\in [0,1]$. As $\exp(Dx)$ is convex on $[0,1]$, we obtain $\exp(D\frac{r}{m})\leq (e^D-1)\frac{r}{m}+1$. Taking $C=(e^D-1)$ proves the claim.
\end{proof}

\begin{proof}[Proof of Proposition \ref{prop-randomsplit-largetau}] 
Let $0<c<1$ and let $G=K +_p I$ be a split graph with $|K|=\frac{n}{\log n}$, $|I|=n$, and $p=1-c'$ (we determine $c' \approx c$ later). So, $|V(G)|=n\left(1+\frac{1}{\log n}\right)$. Let $\cF=\{\nc_I(v): v\in K\}$ denote the collection of non-neighborhoods. By Observation~\ref{obs-splitgraphtau}, $\tau(G)=\tau(\cF)$. We prove $\tau(\cF)$ is large by a similar argument given in the proof of Theorem~\ref{thm-janson-construction}. Put $\delta = \frac{-1}{\log p}$. Fix $T\subseteq I$ with $t\coloneqq |T|= \delta\log n-4\delta \log \log n$. For $F\in \cF$, we have $\pr\{F\cap T=\emptyset\}= p^t$. Since the collection of events $\big\{ \{F\cap T\neq \emptyset \}: F\in \cF\big\}$ is mutually independent, we have\[
\pr\left[\bigwedge_{F\in \cF} \{F\cap T\neq \emptyset \}\right]= (1- p^t)^{|K|}\leq \exp\left(-p^t|K|\right).
\] 
By the union bound, \begin{align*}
\pr\left[ \tau(\cF)\leq t \right]&\leq \binom{n}{t}\exp\left(-p^t|K|\right)\\
&\leq n^t\exp\left(-p^t|K|\right)= \exp\left(-p^t|K| + t\log n\right)\\
 &=\exp\left(\frac{-n}{\log n}p^{\delta\log n-4\delta\log\log n}+ t\log n \right)\\
    &= \exp\left(\frac{-n}{\log n}\exp\left(\delta\log p\log n-4\delta\log p\log\log n\right)+ t\log n\right)\\
    &= \exp\left(-(\log n)^{-4\delta\log p-1}+ t\log n \right)\\
    &= \exp\left(-(\log n)^{3}+ \delta(\log n)^2-4\delta\log n\log\log n \right),
\end{align*}
which tends to 0 as $n\to \infty$. This establishes that $\tau(G)> (1-o(1))\frac{-1}{\log (1-c')}\log n$ whp. We now choose $c'= c(1+\frac{1}{\log n})+\sqrt{\frac{\log n}{n}}$ so that $c'n - \sqrt{n\log n}=c|V(G)|$. Fix $v\in K$. Notice the random variable $|\nc_I(v)|$ is a sum of $n$ Bernoulli random variables with probability $c'$. Applying Hoeffding's inequality of the form $\pr\left(\E[S_n]-S_n \geq \lambda \right) \leq \exp(-2\lambda^2/n)$ with $\lambda=\sqrt{n\log n}$, we deduce that
\[
\pr\left[ |\nc_I(v)| \leq c|V(G)|\right]\leq \exp\left(-2\log n\right) = \frac{1}{n^2}.
\]
By applying the union bound over all $\frac{n}{\log n}$ sets in $\cF$, we conclude that $G$ is $c$-hollow whp. (The maximal independent set $I$ does not cause problems since $|I|=n\geq cn(1+\frac{1}{\log n})$ as $n\to\infty$.) With this choice of $c'$, we also have 
\begin{align*}
    \tau(G)&\geq (1-o(1)) \frac{- \log n}{\log(1-c') }\\
    &= (1-o(1)) \frac{-\log n}{\log\left(1- c(1+\frac{1}{\log n})-\sqrt{\frac{\log n}{n}}\right) }\\
    &= (1-o(1))\frac{-\log n}{\log(1-c)},
\end{align*}
which completes the proof. 
\end{proof}

\bibliographystyle{plain}
\bibliography{refs.bib}

@article{moon1965,
	title={On cliques in graphs},
	author={Moon, John W and Moser, Leo},
	journal={Israel Journal of Mathematics},
	volume={3},
	pages={23--28},
	year={1965},
	publisher={Springer}
}

@article{alon2021hitting,
  title={Hitting all maximum independent sets},
  author={Alon, Noga},
  journal={arXiv preprint arXiv:2103.05998},
  year={2021}
}

@article{cheng2024sublinear,
  title={Sublinear hitting sets for some geometric graphs},
  author={Cheng, Xinbu and Huang, Xinqi and Rong, Mingyuan and Xu, Zixiang},
  journal={arXiv preprint arXiv:2404.10379},
  year={2024}
}

@article{alon1990transversal,
  title={Transversal numbers of uniform hypergraphs},
  author={Alon, Noga},
  journal={Graphs and Combinatorics},
  volume={6},
  number={1},
  pages={1--4},
  year={1990},
  publisher={Springer}
}

@article{tuza1990covering,
  title={Covering all cliques of a graph},
  author={Tuza, Zsolt},
  journal={Discrete Mathematics},
  volume={86},
  number={1-3},
  pages={117--126},
  year={1990},
  publisher={Elsevier}
}

@article{hajebi2024hitting,
  title={Hitting all maximum stable sets in {$P_5$}-free graphs},
  author={Hajebi, Sepehr and Li, Yanjia and Spirkl, Sophie},
  journal={Journal of Combinatorial Theory, Series B},
  volume={165},
  pages={142--163},
  year={2024},
  publisher={Elsevier}
}

@techreport{miller1960problem,
  title={A problem of maximum consistent subsets},
  author={Miller, Raymond E and Muller, David E},
  year={1960},
  institution={IBM Research Report RC-240, JT Watson Research Center, Yorktown Heights, NY}
}

@article{erdos1992covering,
  title={Covering the cliques of a graph with vertices.},
  author={Erd\H{o}s, Paul and Gallai, Tibor and Tuza, Zsolt},
  journal={Discret. Math.},
  volume={108},
  number={1-3},
  pages={279--289},
  year={1992}
}

@book{chung1998erdos,
  title={Erd\H{o}s on Graphs: His Legacy of Unsolved Problems},
  author={Chung, Fan and Graham, Ron},
  year={1998},
  publisher={AK Peters/CRC Press}
}

@article{hajnal1965theorem,
  title={A theorem on $k$-saturated graphs},
  author={Hajnal, Andr{\'a}s},
  journal={Canadian Journal of Mathematics},
  volume={17},
  pages={720--724},
  year={1965},
  publisher={Cambridge University Press}
}

@article{ai2024piercing,
  title={Piercing independent sets in graphs without large induced matching},
  author={Ai, Jiangdong and Liu, Hong and Xu, Zixiang and Zhou, Qiang},
  journal={arXiv preprint arXiv:2403.19737},
  year={2024}
}

@article{goddard2013independent,
  title={Independent domination in graphs: A survey and recent results},
  author={Goddard, Wayne and Henning, Michael A},
  journal={Discrete Mathematics},
  volume={313},
  number={7},
  pages={839--854},
  year={2013},
  publisher={Elsevier}
}

@article{janson1998new,
  title={New versions of {S}uen's correlation inequality},
  author={Janson, Svante},
  journal={Random Structures and Algorithms},
  volume={13},
  number={3-4},
  pages={467--483},
  year={1998}
}

@article{erdos3problems,
  title={Problems and results on set systems and hypergraphs},
  author={Erd\H{o}s, Paul},
  journal={Bolyai Soc. Math. Stud.},
  volume={3},
  pages={217--227},
  year={1991}
}

@article{cheng2025bollobas,
  title={Bollob{\'a}s-{Erd{\H{o}}s}-{Tuza} Conjecture for Graphs With No Induced {$K_{s,t}$}},
  author={Cheng, Xinbu and Xu, Zixiang},
  journal={Journal of Graph Theory},
  volume={109},
  number={4},
  pages={514--517},
  year={2025},
  publisher={Wiley Online Library}
}

@book{alon2016probabilistic,
  title={The probabilistic method},
  author={Alon, Noga and Spencer, Joel H},
  year={2016},
  publisher={John Wiley \& Sons}
}

@article{hoeffding1963probability,
  title={Probability inequalities for sums of bounded random variables},
  author={Hoeffding, Wassily},
  journal={Journal of the American Statistical Association},
  volume={58},
  number={301},
  pages={13--30},
  year={1963},
  publisher={Taylor \& Francis}
}

@article{telelis2005absolute,
  title={Absolute $o(\log m)$ error in approximating random set covering: an average case analysis},
  author={Telelis, Orestis A and Zissimopoulos, Vassilis},
  journal={Information Processing Letters},
  volume={94},
  number={4},
  pages={171--177},
  year={2005},
  publisher={Elsevier}
}

@article{vercellis1984probabilistic,
  title={A probabilistic analysis of the set covering problem},
  author={Vercellis, Carlo},
  journal={Annals of Operations Research},
  volume={1},
  number={3},
  pages={255--271},
  year={1984},
  publisher={Springer}
}

@article {lovasz75,
    AUTHOR = {Lov\'asz, L.},
     TITLE = {On the ratio of optimal integral and fractional covers},
   JOURNAL = {Discrete Math.},
  FJOURNAL = {Discrete Mathematics},
    VOLUME = {13},
      YEAR = {1975},
    NUMBER = {4},
     PAGES = {383--390},
      ISSN = {0012-365X,1872-681X},
   MRCLASS = {05B40},
  MRNUMBER = {384578},
MRREVIEWER = {Torrence\ D.\ Parsons},
       DOI = {10.1016/0012-365X(75)90058-8},
       URL = {https://doi.org/10.1016/0012-365X(75)90058-8},
}
\end{document}